\documentclass[10pt]{amsart}
\usepackage{graphicx}
\usepackage{graphicx}
\usepackage{amssymb}
\usepackage{epstopdf}
\newtheorem{thm}{Theorem}[section]
\newtheorem{cor}[thm]{Corollary}

\newtheorem{prop}[thm]{Proposition}
\theoremstyle{definition}

\theoremstyle{remark}
\newtheorem{rem}[thm]{Remark}
\numberwithin{equation}{section}

\title{Euler characteristic numbers of space-like manifolds}
\author{Bing-Long Chen, Kun Zhang}
\begin{document}
\begin{abstract}
In this note, we prove that if a compact even dimensional manifold $M^{n}$ with negative sectional curvature is homotopic to some compact space-like manifold $N^{n}$, then the Euler characteristic number of $M^{n}$ satisfies $(-1)^{\frac{n}{2}}\chi(M^{n})>0$. We also show that the minimal volume conjecture of Gromov is true for all compact even dimensional space-like manifolds.
\end{abstract}
\maketitle

\section{Introduction}

Let $M^{n}$ be a compact even dimensional Riemannian manifold with negative sectional curvature. A long-standing conjecture due to H.~Hopf\cite{hopf25} in differential geometry asks whether the Euler characteristic number of $M^{n}$ satisfies $(-1)^{\frac{n}{2}}\chi(M^{n})>0$.  When $n=4$, the proof was given by Chern\cite{chern55}~(who attributed the result to Milnor) by showing the integrand of Gauss-Bonnet-Chern is positive. However, when $n=6$, some examples show that the integrand does not have a definite sign in general. On the other hand, Gromov in \cite{gro91} proved that Hopf conjecture is true when the manifold is K\"ahler.

In this note, we will consider the Euler characteristic numbers of a class of real Riemannian manifolds. These manifolds are locally embeddable in Lorentz-Minkowski space $R^{n,1}$. In \cite{zhang09}, such $N^{n}$ is called space-like. More precisely, we call a manifold $(N^{n},g)$~(see \cite{zhang09}) space-like if there exists a symmetric $(0,2)$ tensor $h_{ij}$ such that the following two equations are fulfilled
\begin{equation}\label{ge}
R_{ijkl}=-(h_{ik}h_{jl}-h_{il}h_{jk});
\end{equation}
\begin{equation}\label{ce}
\nabla_{i}h_{jk}=\nabla_{j}h_{ik}.
\end{equation}
Here the sign convention for the Riemann curvature tensor $R_{ijkl}$ is made so that $R_{ijij}$ is positive on sphere. Clearly, a space-like $n$-dimensional submanifold of $R^{n,1}$ satisfies the above two equations (1.1) and (1.2) if we take the tensor $h_{ij}$ to be the second fundamental form. A space-like manifold shares some interesting properties of manifolds with non-positive curvature. For example, it can be shown that the universal cover of a space-like manifold is diffeomorphic to the Euclidean space (see Corollary\ref{cor2.2}). The main result of this note is the following theorem:

\begin{thm}\label{thm1}
Let $M^{n}$ be an even dimensional Riemannian manifolds with negative sectional curvature. Suppose $M^{n}$ is homotopic to some compact space-like manifold $N^{n}$. Then the Euler characteristic number of $M^{n}$ satisfies 
$$(-1)^{\frac{n}{2}}\chi(M^{n})>0.$$
\end{thm}

Note that in the theorem, we do not assume the curvature of space-like manifold $N^{n}$ has a sign. The curvature sign is only imposed on the manifold $M^{n}$. The motivation for the proof of Theorem 1.1 is from \cite{zhang09}, where the second author studied the hyperbolization problem of space-like manifolds by using the so-called intrinsic mean curvature flow. More precisely, it was shown in \cite{zhang09} that if the manifold is compact and $h_{ij}>0$, then the manifold admits a Riemannian metric of negative constant sectional curvature.

Theorem 1.1 follows from a more general result on space-like manifolds satisfying (1.1) and (1.2):

\begin{thm}\label{thm2}
Let $N^{n}$ be an even dimensional compact space-like manifold, then the Euler characteristic number satisfies
$$(-1)^{\frac{n}{2}}\chi(M^{n})>0.$$
The equality holds if and only if the minimal volume of $N^{n}$ is zero.
\end{thm}

According to Gromov\cite{gro82}, the minimal volume of a manifold $N^{n}$ is the infimum of all volumes $vol(N^{n},g')$, where $g'$ ranges over all Riemannian metrics with sectional curvatures satisfying $|K_{g'}|\leq 1$. The minimal volume conjecture of Gromov \cite{gro82} is asking whether there is a number $\varepsilon(n)$ depending only on the dimension $n$ such that min vol$(N^{n})<\varepsilon(n)$ implies min vol~$(N^{n})=0$. This conjecture was already verified by X. C. Rong\cite{rong93} in dimension 4. A byproduct of Theorem \ref{thm2} is 
\begin{cor}
The minimal volume conjecture is true for all compact even dimensional space-like manifolds.
\end{cor}
The proof of above Theorems \ref{thm1} and \ref{thm2} is an elementary application of the mean curvature flow. The detail will be given in the following two sections.

%---------------------------------------------------
\section{Mean curvature flow}

In this section, we assume that $(N^{n},g)$ is a space-like manifold, i.e. there is a tensor $h_{ij}$ such that equation (\ref{ge}) and (\ref{ce}) hold. In \cite{zhang09}, the second author studied the following flow:
 \begin{equation} \label{e2.1}
 \begin{split}
\frac{\partial g_{ij}}{\partial t}&=-2R_{ij}+2h_{ik}h_{jl}g^{kl},\\
\frac{\partial h_{ij}}{\partial t}&= \triangle h_{ij}-R_{im}h_{nj}g^{mn}-R_{jm}h_{ni}g^{mn}\\ & \ \ \ \ +2h_{im}h_{jn}h_{kl}g^{mk}g^{ln}-h_{mn}h_{kl}g^{mk}g^{nl}h_{ij}.
\end{split}
\end{equation}
It was shown in \cite{zhang09}, when $(N^{n},g)$ is compact, equation (\ref{e2.1}) admits a smooth solution for any initial data $(g_0,h_0).$ Moreover, if equations (\ref{ge}) and (\ref{ce}) hold at time $t=0,$ then they also continue to  hold for  time $t>0.$ That is to say, $(N^n,g(t),h(t))$ will remain to be a space-like manifold under the deformation (\ref{e2.1}). In this case,  equation (\ref{e2.1}) may be simplified:
 \begin{equation}\label{e2.2}
\begin{split}
\frac{\partial g_{ij}}{\partial t}&=2Hh_{ij}  \\
\frac{\partial h_{ij}}{\partial t} &=\triangle h_{ij}+2Hh_{im}h_{jn}g^{mn}-|A|^2h_{ij} \\
\end{split}
\end{equation}
 where $H=g^{ij}h_{ij}, |A|^2=g^{ij}g^{kl}h_{ik}h_{jl}.$

We may call equation (\ref{e2.1})  an  intrinsic mean curvature flow. Equations (\ref{ge}) and (\ref{ce}) may be called Gauss and Codazzi equations respectively.

Another approach to solve equation (\ref{e2.1}) or (\ref{e2.2}) is to embed the universal cover $(\tilde{N},\tilde{g})$ into $R^{n,1}$ as a space-like submanifold $\Sigma$ in the usual sense, deform $\Sigma$ in $R^{n,1}$ by the mean curvature flow and prove the mean curvature flow is invariant under the deck transformation.
\begin{prop}\label{prop2.1}
$\tilde{N}$ admits an isometric embedding into $R^{n,1}$ as a space-like submanifold with second fundamental form $h_{ij}$ given in (\ref{ge}) and (\ref{ce}).
\end{prop}
\begin{proof}
By a monodromy argument, there is a smooth isometric immersion $\varphi: (\tilde{N},\tilde{g})\rightarrow R^{n,1}$ with $h_{ij}$ as the second fundamental form. Let $\pi:R^{n,1}\rightarrow R^{n}$ be the projection to an $n$-coordinate plane, and $\psi=\pi\circ\varphi$. Let $g_{0}$ be the Euclidean metric in $R^{n}$, then it is not hard to see $\tilde{g}\leq\psi^{*}g_{0}$. This implies $\psi$ is proper, hence a covering map to $R^{n}$. From this, we know $\varphi$ is an embedding.
\end{proof}
\begin{cor}\label{cor2.2}
The universal cover $\tilde{N}$ is diffeomorphic to the Euclidean space.
\end{cor}

Now let $(N^{n},g)$ be a compact space-like manifold, we deform $(g,h)$ by (\ref{e2.1}) or (\ref{e2.2}).

From (\ref{e2.2}), it is not hard to show
\begin{equation}\label{e2.3}
\begin{split}
\frac{\partial H}{\partial t} &= \triangle H-H|A|^2 \\
\frac{\partial |A|^2}{\partial t}&=\triangle |A|^2-2|\nabla A|^2-2|A|^4
\end{split}
\end{equation}
where $|\nabla A|^2=g^{ij}g^{kl}g^{pq}\nabla_ih_{kp}\nabla_{j}h_{lq}.$

Based on (\ref{e2.3}), it can be shown that the solution $(g(t),h(t))$ of equation (\ref{e2.1}) (or (\ref{e2.2})) always exists for all time $0\leq t<\infty,$  and $h_{ij}$ satisfies the estimate
\begin{equation}\label{2.4}
  0\leq|A|^2\leq\frac{1}{2t+1/|A|_{max}(0)}.
\end{equation}

In the following propositions, we assume the dimension $n$ to be even.\\

First, we derive two monotonicity formulas for the intrinsic mean curvature flow. We have to mention that all the quantities in the following propositions involving the norm and the volume element $dv$ are computed at time $t$ with the evolving metric $g(t).$

\begin{prop}\label{prop2.3} 
\begin{equation*}
    \frac{d}{dt}\int_{N^{n}}H^{n}dv=-n(n-1)\int_{N^{n}}|\nabla H|^2H^{n-2}dv-n\int_{N^{n}}H^{n}|h_{ij}-\frac{H}{n}g_{ij}|^2dv\leq 0.
\end{equation*}
and
\begin{equation*}
\begin{split}
\frac{d}{dt}\int_{N^{n}}|A|^{n}dv
=&-(\frac{n}{2}-1)\frac{n}{2}\int_{N^{n}}|A|^{(n-4)}|\nabla|A|^2|^2dv
-n\int_{N^{n}}|A|^{n-2}|\nabla A|^2dv\\
&-n\int_{N^{n}}|A|^{n}|h_{ij}-\frac{H}{n}g_{ij}|^2dv\leq 0.
\end{split}
\end{equation*}
\end{prop}
\begin{proof}
The proof is direct calculations by using equations (\ref{e2.3}).
\end{proof}

\begin{prop}\label{prop2.4}
There are constant $C_0>0$  and a sequence of times $t_k\rightarrow \infty$ such that
\begin{equation}\label{2.5}
0\leq\int_{N^{n}}|A|^{n}dv<C_0 ,
\end{equation}
for all time $t\geq 0,$ and
\begin{equation}\label{2.6}
t_k\cdot \int_{N^{n}}|A|^{n}|h_{ij}-\frac{H}{n}g_{ij}|^2dv\mid_{t=t_k}\rightarrow 0 \quad as \quad k\rightarrow \infty.
\end{equation}
\end{prop}
\begin{proof}
Integrating the time in the first formula of Proposition \ref{prop2.3} from $-\infty$ to $0$, we have
\begin{equation}\label{2.7}\int_{-\infty}^{0}\int_{N^{n}}|A|^{n}|h_{ij}-\frac{H}{n}g_{ij}|^2dV\leq C_0.\end{equation}
If (\ref{2.6}) does not hold, then there is $C>0$ such that for all $t>0,$ we have $$\int_{N^{n}}|A|^{n}|h_{ij}-\frac{H}{n}g_{ij}|^2dv>\frac{1}{t+C},$$ which is a contradiction with (\ref{2.7}).
\end{proof}

Since
$$
  \frac{d}{dt}Vol(N^n,t)=\int_{N^n}H^2dv
  \leq \Big(\int_{N^n}|H|^ndv\Big)^{\frac{2}{n}}(Vol(N^n,t))^{1-\frac{2}{n}},
$$
we have \begin{equation}\label{2.8'}
  \frac{d}{dt}Vol(N^n,t)^{\frac{2}{n}}\leq \frac{2}{n}(\int_{N^n}|H|^ndv)^{\frac{2}{n}}.\end{equation}
 This implies
\begin{prop}
There is a  constant $C_1>0$ such that for all $t>0$
\begin{equation}\label{e2.9}
  Vol(N^n,t)\leq C_1(t+1)^{\frac{n}{2}},
\end{equation}
\begin{equation}\label{2.10}
  \frac{1}{1+t}\int_{N^n}|A|^{n-2}dV\leq C_1.
\end{equation}
Moreover when n is even
\begin{equation}\label{e2.8}
 \limsup_{t\rightarrow \infty} \frac{Vol(N^n,t)}{(1+t)^{\frac{n}{2}}}\leq (\frac{2}{n})^{\frac{n}{2}}\lim_{t\rightarrow \infty} \int_{N^n}|H|^ndv.%\mid_{t=t_k}.
\end{equation}
\end{prop}
\begin{proof}
 From \eqref{2.8'} we have
\begin{equation*}
\begin{split}
   \frac{Vol(N^n,t)^{\frac{2}{n}}-Vol(N^n,0)^{\frac{2}{n}}}{t+1}&\leq\frac{2}{n}\frac{\int_0^t (\int_{N^n}|H|^ndv)^{\frac{2}{n}}dt}{t+1}.
\end{split}
\end{equation*}
Since $n$ is even, and $\int_{N^n}|H|^ndv$ is monotonically decreasing from the first formula in Proposition 2.1,
 we know
\begin{equation*}
  \limsup_{t\rightarrow \infty} \frac{Vol(N^n,t)^{\frac{2}{n}}}{1+t}\leq \frac{2}{n}\lim_{t\rightarrow \infty} (\int_{N^n}|H|^ndv)^\frac{2}{n}.
\end{equation*}
\end{proof}

%-----------------------------------------------------------------
\section{Proof of the main theorem}

To prove Theorem \ref{thm2}, we may assume $N^n$ is orientable.  It is well-known that by Gauss-Bonnet-Chern theorem, the Euler Characteristic number $\chi(N^n)$ may be expressed as a curvature integral (see Chern \cite{chern44}):
\begin{equation}\label{3.0'}
\int_{N^n} \Omega=\chi(N^n)
\end{equation}
where $$\Omega=\frac{1}{2^n\pi^{\frac{n}{2}}(\frac{n}{2})!}\sum \varepsilon^{i_1,i_2,\cdots,i_n}\Omega_{i_1i_2}\wedge\cdots \Omega_{i_{n-1}i_n},$$ and $\Omega_{i_1i_2}=R_{i_1i_2 j_1j_2}dx^{j_1}\wedge dx^{j_2}$ is the curvature form.
From equation (\ref{ge}) and direct calculations,  it follows \begin{equation}\label{3.0}
\Omega=(-1)^{\frac{n}{2}}\frac{\Gamma(\frac{n+1}{2})}{\pi^{\frac{n+1}{2}}}\frac{\det(h)}{\det(g)}dv=(-1)^{\frac{n}{2}}
\frac{2}{vol(S^{n})}\frac{\det(h)}{\det(g)}dv,
 \end{equation}
 and
 \begin{equation}\label{3.00}
\chi(N^n)=(-1)^{\frac{n}{2}}
\frac{2}{vol(S^{n})}\int_{N^n}\frac{\det(h)}{\det(g)}dv.
 \end{equation}

 We remark that equation (\ref{3.00}) holds for any $t>0,$ since equation (\ref{ge}) holds for any time $t>0.$

 For any fixed $p\in N^{n}$, choose an orthonormal frame $e_i,i=1,2,\cdots,n$ such that $h_{ij}$ is diagonalied in this frame, i.e. $h_{ij}=\lambda_{i}\delta_{ij}.$  Then we have
$$|\frac{\det(h)}{\det(g)}-(\frac{H}{n})^{n}|\leq|\lambda_1\cdots\lambda_{n}-(\frac{H}{n})^{n}|\leq n|A|^{n-1}|h_{ij}-\frac{H}{n}g_{ij}|.$$

Hence
$$
\begin{aligned}
    &\int_{N^{n}}|\frac{\det(h)}{\det(g)}-(\frac{H}{n})^{n}|dv\\
    &\leq n\int_{N^{n}}|A|^{n-1}|h_{ij}-\frac{H}{n}g_{ij}|dv\\
    &\leq n(\int_{N^{n}}|A|^{n}|h_{ij}-\frac{H}{n}g_{ij}|^2dv)^{\frac{1}{2}}
        \cdot(\int_{N^{n}}|A|^{n-2}dv)^{\frac{1}{2}}.\\
\end{aligned}
$$
Let $t_k$ be the time sequence chosen in  Proposition \ref{prop2.4}, it follows  from
(\ref{2.6}) and (\ref{2.10}) that at $t=t_k:$
\begin{equation}\int_{N^{n}}|A|^{n}|h_{ij}-\frac{H}{n}g_{ij}|^2dv
        \cdot\int_{N^{n}}|A|^{n-2}dv\mid_{t=t_k}\rightarrow 0
\end{equation}
as $k\rightarrow \infty.$ This implies
\begin{equation}\label{3.5}
\lim_{k\rightarrow \infty}\int_{N^{n}}|\frac{\det(h)}{\det(g)}-(\frac{H}{n})^{n}|dv\mid_{t=t_k}=0.
\end{equation}

Combining (\ref{3.00}) and (\ref{3.5}), we have
\begin{equation}\label{3.6}
(-1)^{\frac{n}{2}}\chi(N^n)=\lim_{k\rightarrow \infty}\frac{2}{vol(S^{n})}\int_{N^{n}}(\frac{H}{n})^{n}dv\mid_{t=t_k}.
\end{equation}
Because $n$ is even, we know $(-1)^{\frac{n}{2}}\chi(N^n)\geq 0.$ This finishes the main part of Theorem \ref{thm2}.

Clearly, if the minimal volume of $N^n$ is zero, then the Euler characteristic number $\chi(N^n)$ is zero. This follows directly from the Gauss-Bonnet-Chern formula (\ref{3.0'}). To see the converse, let $\chi(N^n)=0,$ from (\ref{3.6}) we have
\begin{equation}\label{3.7}
\lim_{k\rightarrow \infty}\int_{N^{n}}H^{n}dv\mid_{t=t_k}=0.
\end{equation}
Combining (\ref{e2.8}), it implies

\begin{equation}\label{e3.8}
 \limsup_{k\rightarrow \infty} \frac{Vol(M,t_k)}{(1+t_k)^{\frac{n}{2}}}=0.
\end{equation}

Note that (\ref{2.4}) and (\ref{ge}) implies $|Rm|(g_{t_k})\leq C t_{k}^{-1}.$ So $\frac{Cg_{t_k}}{{t_k}}$ is a sequence of metrics with sectional curvatures satisfying $|K_{Cg_{t_k}}|\leq 1$ but the volumes converges to zero as $k\rightarrow \infty.$ This shows the minimal volume of $N^n$ is zero. The proof of Theorem \ref{thm2} is completed.

To prove Theorem \ref{thm1}, we recall a result of Gromov\cite{gro82}: the simplicial volume of a compact manifold $X^n$ with negative sectional curvature is positive.  In our case, we have the simplicial volume of $M^n$ is positive, so is $N^n$ by the homotopic invariance of simplicial volume. In paper \cite{gro82}, Gromov  proved that the minimal volume is always bounded from below by the simplicial volume multiplied by a constant depending only on the dimension. Theorem \ref{thm1} follows from this result and Theorem \ref{thm2}. Finally, we mention one corollary:
\begin{cor}
Let $(M^{n},g,h)$ be an even-dimensional compact space-like manifold. Then
\begin{equation}
\frac{1}{vol(S^{n})}\int_{M^{n}}|H|^{n}\geq(-1)^{\frac{n}{2}}\frac{n^{n}}{2}\chi(M^n).
\end{equation}
Equality holds if and only if either $(M^{n},g,h)$ is hyperbolic or flat.
\end{cor}

\begin{rem}
It is desirable to generalize Theorem \ref{thm2} to higher codimensional case. Namely, we may consider the manifold $(M^{n},g)$ which is locally embeddable as a space-like higher codimensional submanifold of $R^{n,m}$. In this case, a formula similar to the first one in Proposition \ref{prop2.3} can still hold:
\begin{equation}
\frac{d}{dt}\int_{M^{n}}|H|^{n}_{+}\leq0,
\end{equation}
where $|H|^{n}_{+}$ is the absolute value of the squared norm of the (time-like) mean curvature vector. This in particular implies
\end{rem}
\begin{equation}
  \limsup_{t\rightarrow \infty} \frac{Vol(M^n,t)^{\frac{2}{n}}}{1+t}\leq \frac{2}{n}\lim_{t\rightarrow \infty} (\int_{M^n}|H|^{n}_{+}dv)^\frac{2}{n}<\infty.
\end{equation}•

{\bf Acknowledgements} The first author is partially supported by NSFC 11025107, the second author by NSFC 11301190.

%\newpage

\end{document}